\documentclass[letterpaper, 10 pt, journal, twoside]{IEEEtran}
\IEEEoverridecommandlockouts                              
                                                          
\usepackage[pdftex]{graphicx}
\graphicspath{{./Figs/}}
\DeclareGraphicsExtensions{.pdf}

\usepackage{cite, multirow, booktabs,colortbl}

\usepackage[cmex10]{amsmath}
\usepackage{amsfonts,amssymb,amsthm}
\usepackage[caption=false,font=footnotesize]{subfig}
\usepackage[hidelinks]{hyperref}
\usepackage{dsfont}

\DeclareMathOperator{\diag}{diag}
\DeclareMathOperator{\rank}{rank}
\DeclareMathOperator{\range}{range}
\DeclareMathOperator{\re}{Re}
\DeclareMathOperator{\im}{Im}
\DeclareMathOperator{\calM}{\mathcal{M}}
\DeclareMathOperator{\calN}{\mathcal{N}}
\DeclareMathOperator{\calF}{\mathcal{F}}

\newtheorem{thm}{Theorem}
\newtheorem{definition}{Definition}

\newtheorem{lemma}{Lemma}
\newtheorem{prop}{Proposition}
\newtheorem{assumption}{Assumption}
\theoremstyle{remark}
\newtheorem{example}{Example}
\newtheorem{remark}{Remark}

\title{Generic Existence of Unique Lagrange Multipliers \\ in AC Optimal Power Flow}

\author{Adrian Hauswirth, Saverio Bolognani, Gabriela Hug, and Florian D\"orfler
\thanks{The authors are with the Department of Information Technology and Electrical Engineering, ETH Zurich,
				8092 Zurich, Switzerland. Email:
				{\tt\small \{hadrian,bsaverio,ghug,dorfler\}@ethz.ch}. This work was supported by ETH Zurich funds and the SNF AP Energy Grant \#160573.}}
\begin{document}

\maketitle
\thispagestyle{empty}
\pagestyle{empty}
\begin{abstract}
Solutions to nonlinear, nonconvex optimization problems can fail to satisfy the KKT optimality conditions even when they are optimal. This is due to the fact that unless constraint qualifications (CQ) are satisfied, Lagrange multipliers may fail to exist. Even if the KKT conditions are applicable, the multipliers may not be unique. These possibilities also affect AC optimal power flow (OPF) problems which are routinely solved in power systems planning, scheduling and operations. The complex structure -- in particular the presence of the nonlinear power flow equations which naturally exhibit a structural degeneracy -- make any attempt to establish CQs for the entire class of problems very challenging.
In this paper, we resort to tools from differential topology to show that for AC OPF problems in various contexts the linear independence constraint qualification is satisfied almost certainly, thus effectively obviating the usual assumption on CQs. Consequently, for any local optimizer there generically exists a unique set of multipliers that satisfy the KKT conditions.
\end{abstract}


\section{Introduction}
\IEEEPARstart{T}{he} Karush-Kuhn-Tucker (KKT) conditions are the most widely used tool to study the optimality of a solution to a constrained optimization problem. They state that for every optimizer there exists a set of Lagrange multipliers that meet certain algebraic conditions.
However, for an optimizer the KKT conditions and hence the existence of (unique) Lagrange multipliers hold only if the active constraints at that point are well-behaved. This quality is captured by \emph{constraint qualifications} (CQ). Different CQs that imply the applicability of the KKT conditions have been studied~\cite{bazaraa_nonlinear_2006, wachsmuth_licq_2013}.
For nonlinear, nonconvex problems CQs can be hard to verify a priori. Hence, CQs are often stated as technical assumptions, and their restrictiveness is rarely discussed.

In power systems, the question about the existence and uniqueness of Lagrange multipliers has mostly been ignored. This applies in particular to the extensive category of AC optimal power flow (OPF) problems~\cite{frank_introduction_2016, frank_optimal_2012, capitanescu_critical_2016, huneault_survey_1991, momoh_review_1999}, i.e., the class of problems that incorporate nonlinear power flow (PF) equations as constraints, even though the KKT optimality conditions have always been an integral part in the study of AC OPF problems~\cite{dommel_optimal_1968, burchett_large_1982, peschon_optimal_1972, alsac_optimal_1974}. In particular, they are the cornerstone of different numerical methods such as primal-dual interior point and Newton SQP~\cite{jabr_primal-dual_2002, de_sousa_optimal_2012, liu_extended_2002, soler_modified_2012, sousa_robust_2011, kourounis_towards_2018, wang_computational_2007}. 
Furthermore, recent methods in distributed online optimization for power systems often require the communication of Lagrange multipliers~\cite{molzahn_survey_2017, dallanese_photovoltaic_2016, bolognani_distributed_2013-1, bolognani_distributed_2015}. Finally, Lagrange multipliers play an important role in market applications where they correspond to price signals~\cite{conejo_locational_2005, baughman_real-time_1991}. While the existence of Lagrange multipliers is sufficient for most optimization methods to converge, market applications in particular require uniqueness, since systematically ambiguous prices would raise doubts about the fairness of a pricing scheme.

All of these works assume the applicability of the KKT conditions either explicitly or implicitly, but do not address the restrictiveness of this assumption. 
An exception is~\cite{almeida_critical_1996} which studies \emph{critical} cases for which OPF solvers can fail.

In this paper, we show that for prototypical AC OPF problems the \emph{linear independence CQ} (LICQ) holds generically and unique Lagrange multipliers exist for all minima generically.
For this, we employ a result based on \emph{Thom's Transversality Theorem} from differential topology that has previously been established for more abstract optimization problems~\cite{spingarn_optimality_1982}.
We elaborate a more accessible result and provide the necessary context for its immediate application in power systems optimization.
Although there exist weaker CQs than LICQ and more general optimality conditions than KKT, our result shows that for practical purposes there is no need to resort to these more sophisticated tools.

The main result admits a probabilistic interpretation and hints at an important application of our result: If an AC OPF problem is randomly sampled from a suitably parametrized class of problems, then the claim that LICQ holds generically amounts to saying that LICQ holds with probability one. For this to hold, the type of parametrization needs to be of high enough dimension for Thom's theorem to be applicable. 
We will see that due to structural degeneracies in the AC power flow equations this is not automatically satisfied. However, for perturbations in loads and network parameters the result can be applied and unique Lagrange multipliers exist almost certainly.
We believe that this insight is particularly important for mission-critical online power system applications that require strong a priori guarantees for AC OPF problems.

The rest of the paper is structured as follows: Section~\ref{sec:prelim} reviews the KKT conditions and CQs, and it also establishes the requirements for genericity of CQs. Section~\ref{sec:opf} introduces a standard class of AC OPF problems and provides examples where CQs fail to hold. Section~\ref{sec:pert} identifies classes of perturbations for AC OPF that guarantee genericity of LICQ\@.

\section{Preliminaries}\label{sec:prelim}

\subsection{Notation}

By $\mathbb{R}^n_+$ we denote the non-negative orthant of $\mathbb{R}^n$ and $a \leq b$ for $a,b \in \mathbb{R}^n$ denotes a componentwise inequality. For a map $f: \mathbb{R}^n \rightarrow \mathbb{R}^m$, the $m \times n$-matrix $\nabla f(x)$ denotes the Jacobian of $f$ at $x$ and $\nabla_{x'} f(x)$ denotes the submatrix of partial derivatives only with respect to the variable $x'$. The \emph{rank of $f$ at $x$} is given by the rank of $\nabla f(x)$. If $m=1$, we call $\nabla f$ the \emph{gradient of $f$}. A map $f: \mathbb{R}^n \rightarrow \mathbb{R}^m$ is of \emph{class $C^r$} if it is $r$-times continuously differentiable  on its domain.

\subsection{Optimality Conditions}

In what follows, we review CQs and optimality conditions for a nonlinear, nonconvex optimization problem of the form
\begin{equation} \label{eq:opt} \tag{P}
\begin{aligned}
		\underset{x \in \mathbb{R}^{n}}{\text{minimize}} \quad & f(x) &\\
		\text{subject to}\quad & h(x) = 0 \qquad g(x) \leq 0
\end{aligned}
\end{equation}
where $h: \mathbb{R}^n \rightarrow \mathbb{R}^I$, $g: \mathbb{R}^n \rightarrow \mathbb{R}^J$ and the cost function $f: \mathbb{R}^n \rightarrow \mathbb{R}$ are continuously differentiable (i.e., $C^1$). Let $\mathbf{J}(x) := \{j \,|\, g_j(x) = 0\}$ denote the set of active constraints at $x \in \mathbb{R}^{n}$ and $g_{\mathbf{J}(x)}(\cdot) := \left[ g_j(\cdot) \right]_{j \in \mathbf{J}(x)}$ the vector obtained from stacking all active inequality constraint functions at~$x$.

\begin{definition} [LICQ]\label{def:licq}
Let $\widehat{x} \in \mathbb{R}^{n}$ be a feasible point of~\eqref{eq:opt}. We say that the \emph{linear independence constraint qualification (LICQ) holds at $\widehat{x}$} if
\begin{equation*}
		\rank \begin{bmatrix}
		\nabla h(\widehat{x}) \\ \nabla g_{\mathbf{J}(\widehat{x})}(\widehat{x})
		\end{bmatrix} = I + |\mathbf{J}(\widehat{x})| \,.
\end{equation*}
\end{definition}

Let the \emph{Lagrangian of~\eqref{eq:opt}} be given as
\begin{equation*}
		L(x, \lambda, \mu) := f(x) + \lambda^T h(x) + \mu^T g(x)
\end{equation*}
for $\lambda \in \mathbb{R}^I$ and $\mu \in \mathbb{R}^J_+$, then
\begin{equation*}
		\nabla_x L(\widehat{x}, \widehat{\lambda}, \widehat{\mu}) = \nabla f(\widehat{x}) + \widehat{\lambda}^T \nabla h(\widehat{x}) + \widehat{\mu}^T \nabla g(\widehat{x})
\end{equation*}

With these notions we can recall a standard optimality condition for nonlinear optimization problems~\cite[11.8]{luenberger_linear_1984}.

\begin{thm}[KKT\@; first-order necessary conditions]\label{thm:first-order}
Let $\widehat{x} \in \mathbb{R}^n$ solve~\eqref{eq:opt}, and assume that the LICQ holds at $\widehat{x}$. Then, there exist $\widehat{\lambda} \in \mathbb{R}^I$ and $\widehat{\mu} \in \mathbb{R}^J_+$ such that
\begin{equation} \label{eq:first-order}
		\nabla_x L(\widehat{x}, \widehat{\lambda}, \widehat{\mu}) = 0 \  \text{and} \  \mu_j = 0 \  \text{for all} \  j \notin \mathbf{J}(\widehat{x}) \,.
\end{equation}
Furthermore, $\widehat{\lambda} \in \mathbb{R}^I$ and $\widehat{\mu} \in \mathbb{R}^J_+$ are unique.
\end{thm}
Theorem~\ref{thm:first-order} concisely describes \emph{stationarity} (of the Lagrangian), (primal/dual) \emph{feasibility} and \emph{complementary slackness} which are often associated with the KKT conditions~\cite{bazaraa_nonlinear_2006}.

Weaker constraint qualifications than the LICQ exist under which Theorem~\ref{thm:first-order} holds~\cite{bazaraa_nonlinear_2006, clarke_optimization_1990}.
However, the uniqueness of Lagrange multipliers is strongly related to the LICQ and does not in general hold under weaker CQs~\cite{kyparisis_uniqueness_1985, wachsmuth_licq_2013}.

\begin{remark}
		For convex problems~\eqref{eq:first-order} in Theorem~\ref{thm:first-order} is necessary and sufficient for global optimality~\cite{bazaraa_nonlinear_2006}. For non-convex problems, second-order optimality conditions have to be taken into account to certify (local) optimality~\cite{luenberger_linear_1984}.
		Also note that, in convex optimization \emph{Slater's condition} is the most widely-used CQ rather than LICQ, as it guarantees strong duality but not necessarily uniqueness. 
\end{remark}

\subsection{Genericity of LICQ}

Recall that a subset $S \subset \mathbb{R}^n$ has \emph{measure zero} if for every $\epsilon > 0$, $S$ can be covered by a countable family of $n$-cubes, the sum of whose measures is less than $\epsilon$. Informally, $S$ has measure zero if its $n$-dimensional volume is zero. If a property holds everywhere except on a set of measure zero, we say that it holds \emph{generically} or \emph{almost everywhere}.

Consider a parametrized optimization problem of the form
\begin{equation} \label{eq:opt_param} \tag{$\text{P}_\xi$}
\begin{aligned}
		\underset{x \in \mathbb{R}^n}{\text{minimize}} \quad & f(x) &\\
		\text{subject to}\quad & \widetilde{h}(x, \xi) = 0 \qquad && \widetilde{g}(x, \xi) \leq 0 \\
													 & h(x) = 0 \qquad && g(x) \leq 0
\end{aligned}
\end{equation}
where $\xi \in \Xi$ is a problem parameter and $\Xi$ is a nonempty, open subset of $\mathbb{R}^{k}$ with non-zero measure. In this context, $h: \mathbb{R}^n \rightarrow \mathbb{R}^{I}$ and $g: \mathbb{R}^n \rightarrow \mathbb{R}^{J}$ are referred to as \emph{fixed constraints}~\cite{spingarn_optimality_1982}. We define the \emph{fixed feasible set} as
\begin{align*}
		\mathcal{X} := \{x \, | \, h(x)= 0, \, g(x) \leq 0 \} \,.
\end{align*}
The functions $\widetilde{h}:\mathbb{R}^n\times \Xi \rightarrow \mathbb{R}^{\widetilde{I}}$ and $\widetilde{g}:\mathbb{R}^n\times \Xi \rightarrow \mathbb{R}^{\widetilde{J}}$ define \emph{variable constraints} due to their dependence on~$\xi$.

\begin{assumption}[Fixed LICQ]\label{ass1}
Given~\eqref{eq:opt_param}, assume that for all $x \in \mathcal{X}$ the LICQ holds with respect to $g$ and $h$, i.e.,
\begin{equation*}
		\rank \begin{bmatrix}
		\nabla h(x) \\ \nabla g_{\mathbf{J}(x)}(x)
		\end{bmatrix} = I + |\mathbf{J}(x)| \,.
\end{equation*}
\end{assumption}

The following theorem states that all feasible points of~\eqref{eq:opt_param} satisfy the LICQ for almost all values of $\xi$ and $x \in \mathcal{X}$ as long as the parametrized constraints of~\eqref{eq:opt_param} span a space of high enough dimension as function of $\xi$.

\begin{thm}[Genericity of LICQ]\label{thm:base1}
Consider the problem~\eqref{eq:opt_param} with $f, \widetilde{g}, \widetilde{h}, g$ and $h$ of class $C^r$ for all $x \in \mathcal{X}$ and $\xi \in \Xi$ with $r>\max(0, n - I - \widetilde{I})$ and let Assumption~\ref{ass1} be satisfied.
If the map
\begin{equation*}
		\xi \mapsto \left( \widetilde{h}(x, \xi) , \, \widetilde{g}(x, \xi) \right)
\end{equation*}
has rank $\widetilde{I}+\widetilde{J}$ for every $x \in \mathcal{X}$ and every $\xi \in \Xi$, then the LICQ holds for almost every $\widehat{\xi} \in \Xi$ and for every $\widehat{x}$ that is feasible for~\eqref{eq:opt_param}.
\end{thm}

From a geometric perspective, LICQ states that the \emph{faces}, defined by the individual constraints that delimit the feasible set must not intersect tangentially at a given point. However, tangency is ``fragile'' in the sense that a small perturbation can eliminate it. Theorem~\ref{thm:base1} formalizes this idea.

Theorem~\ref{thm:base1} isolates a special case of~\cite[Prop~3.6]{spingarn_optimality_1982} which considers a more abstract setting and a stronger notion of genericity. We provide a short self-contained proof in the appendix that captures the main feature of interest.

\section{Power System Model}\label{sec:opf}

We now turn to the study of CQs for AC OPF\@. After introducing the standard AC PF equations, we give examples showing that LICQ does not necessarily hold for AC OPF\@.

Consider a power network consisting of $N$ nodes and $M$ lines. Each node $k$ has an associated voltage $u_k = v_k e^{j\theta_k}$ with $j = \sqrt{-1}$, a fixed load $p^\text{L}_k + j q^\text{L}_k$, variable generation $p^\text{G}_k + j q^\text{G}_k$ as well as a nodal shunt admittance $y^{\text{sh}}_k = g^{\text{sh}}_k + jb^{\text{sh}}_k$.
The set of nodes connected to $k$ through a line are denoted by $\mathbf{N}(k)$. Each line between two nodes $k$ and $l$ is represented by a standard $\Pi$-model with a series admittance $y_{kl} = g_{kl} + jb_{kl}$ and a shunt admittance $y^{\text{sh}}_{kl} = g^{\text{sh}}_{kl} + jb^{\text{sh}}_{kl}$.

The nodal admittance matrix $Y$ is defined as
\begin{equation*}
		[Y]_{kl} = \begin{cases}
			y^{\text{sh}}_k +
			\sum_{i \in \mathbf{N}(k)}{y_{ki} + \frac{1}{2} y^{\text{sh}}_{ki}} \qquad & \text{for } k=l \\
			- y_{kl} \qquad & \text{for } l \in \mathbf{N}(k) \\
			0 \qquad & \text{otherwise} \\
		\end{cases}
\end{equation*}
The AC power flow equations~\cite{frank_introduction_2016} can be written as
\begin{equation*}
		F(x) = \begin{bmatrix}
						 p^\text{G} - p^\text{L} - \re \{ \diag(u) (Y u)^* \} \\
						 q^\text{G} - q^\text{L} - \im \{ \diag(u) (Y u)^* \} \\
				\end{bmatrix} = 0 \,.
\end{equation*}
where $(\cdot)^*$ denotes the complex conjugate, and we use
\begin{equation} \label{eq:state_def}
		x = \left( p^\text{G} , \, q^\text{G}, \, v , \, \theta \right) \in \mathbb{R}^{4N}
		\,.
\end{equation}
Note that $F: \mathbb{R}^{4N} \rightarrow \mathbb{R}^{2N}$ is smooth (i.e.,\,$C^\infty$) in\,$x$.

Additional constraints are generally introduced to limit generation output and voltage magnitudes, enforce line limits, etc. These \emph{operational constraints} can be formulated as
\begin{align*}
		h(x) = 0 \qquad g(x) \leq 0
\end{align*}
where $h$ and $g$ map $\mathbb{R}^{4N}$ to $\mathbb{R}^I$ and $\mathbb{R}^J$ respectively. When applying Theorem~\ref{thm:base1}, these constraints will be the fixed constraints, i.e., not subject to perturbations. As such, $g$ and $h$ will need to satisfy Assumption~\ref{ass1}, that is, the operational constraints in isolation satisfy the LICQ\@. This can often be easily checked, unlike for the full AC OPF problem.

Given any continuously differentiable cost function $f: \mathbb{R}^{4N} \rightarrow \mathbb{R}$ we define the AC OPF problem
\begin{equation} \label{eq:opf} \tag{OPF}
\begin{aligned}
		\underset{x \in \mathbb{R}^{4N}}{\text{minimize}} \quad & f(x) &\\
		\text{subject to}\quad & F(x) = 0 \\ &h(x) = 0 \qquad g(x) \leq 0
\end{aligned}
\end{equation}
We are interested in whether all (local) solutions to~\eqref{eq:opf} satisfy the LICQ and consequently whether unique multipliers exist that satisfy the KKT conditions at each solution.
We present two examples which show that the type of constraints encountered in AC OPF have the potential to produce degeneracies resulting from a breakdown of LICQ\@. Both examples have been designed to provide insight and be analytically tractable, rather than being realistic.

\begin{example}\label{ex:pf1} We show that even in a simple economic dispatch problem, the multipliers associated with the power balance at a bus (also referred to as \emph{nodal price}) can be ambiguous.

Consider a 2-bus power system with a lossless line connecting buses 1 and 2 with unit susceptance $y_{12} = -1j $. A~generator is connected to bus 1. A flexible load is situated at bus 2 modeled as negative generation with constant power factor $ \frac{q_2}{p_2} = \alpha$ and $\alpha > 0$. Assume that no shunt admittances are present (i.e., $y^{sh}_1 = y^{sh}_2 = y^{sh}_{12} = 0$). Furthermore, we enforce an upper voltage limit $v_2 \leq \overline{v}$ at bus~2 for which we choose $\overline{v} := \sqrt{\alpha^2 + 1}$. Hence, the AC power flow equations and operational constraints can be written as
\begin{align} \label{eq:ex3_pf}
F(x) &= \begin{bmatrix}
		p_1 + v_2 \sin \theta_2 \\
		p_2 - v_2 \sin \theta_2 \\
		q_1 + v_2 \cos \theta_2 - 1 \\
		q_2 + v_2 \cos \theta_2 - v_2^2 \\
\end{bmatrix} = 0  \\
h(x) &= \begin{bmatrix}
		q_2 - \alpha p_2
\end{bmatrix} = 0 \\
g(x) &= \begin{bmatrix}
		v_2 - \overline{v}
\end{bmatrix} = 0 \, ,  
\end{align}
where, for simplicity, we have eliminated $\theta_1 = 0$ and $v_1 = 1$ of the slack bus.

Hence, the partial derivatives of the constraint functions with respect to $p_1, p_2, q_1, q_2$ and $v_2$ are given by
\begingroup 
\setlength\arraycolsep{4pt}
\begin{align} \label{eq:ex3_deriv}
		\begin{bmatrix}
			\nabla F \\ \nabla h \\ \nabla g
		\end{bmatrix} &=
		\begin{bmatrix}
				1 & 0 & 0 & 0 & \sin \theta_2 & v_2 \cos \theta_2 \\
				 0 & 1 & 0 & 0 & - \sin \theta_2 & -v_2 \cos \theta_2 \\
				 0 & 0 & 1 & 0 & \cos \theta_2 & - v_2 \sin \theta_2 \\
				 0 & 0 & 0 & 1 & \cos \theta_2 - 2 v_2 & - v_2 \sin \theta_2 \\
				0 & -\alpha & 0 & 1 & 0 & 0 \\
				0 & 0 & 0 & 0 & 1 & 0
		\end{bmatrix}.
\end{align}
\endgroup
Notice that for $\sin \theta_2 = \alpha \cos \theta_2 $ the matrix~\eqref{eq:ex3_deriv} does not have full rank because the second and the last three rows become linearly dependent. This holds in particular for the feasible operating point $x^\star$ given by
\begin{align*}
		v^\star_2 = \sqrt{\alpha^2 + 1} \qquad \theta^\star_2 = \arctan \alpha \\
		p^\star_2 = -p^\star_1 = -\alpha \qquad q^\star_1 = 0 \qquad q^\star_2 = \alpha^2 \,.
\end{align*}
Note that the voltage constraint $v_2^\star - \overline{v} \leq 0$ is active. It follows that LICQ does not hold at $x^\star$ and strictly speaking Theorem~\ref{thm:first-order} is not applicable.\footnote{However, weaker constraint qualifications than LICQ hold and the KKT conditions can be applied albeit without the uniqueness guarantee~\cite{bazaraa_nonlinear_2006}.} Nevertheless, $x^\star$ minimizes the cost function
\begin{align*}
		f(x) = \frac{1}{2} p_1^2 + \alpha p_2
\end{align*}
subject to $F(x) = 0, h(x)=0$ and $g(x) \leq 0$.
The gradient of $f$ at $x^\star$ is given by
$
		\nabla f(x^\star) = \begin{bmatrix}
		\alpha & \alpha & 0 & 0 & 0 & 0
		\end{bmatrix} $.
Despite LICQ not holding, we may solve
\begin{align*}
	 0 = \nabla f(x^\star) + \kappa^T \nabla F(x^\star) + \lambda \nabla h(x^\star) + \mu \nabla g(x^\star)
\end{align*}
with $\mu \geq 0$ which yields
\begin{align*}
	 \begin{bmatrix} \kappa^T & \lambda & \mu \end{bmatrix} = \begin{bmatrix} - \alpha & - \alpha & 0 & 0 & 0 & 0 \end{bmatrix} + \zeta w
\end{align*}
with $\zeta \geq 0$, and
$w = \begin{bmatrix} 0 & -\alpha & 0 & 1 & -1 & \sqrt{\alpha^2 + 1} \end{bmatrix}$.

We conclude that the multiplier for the power balance at bus~2, i.e., the negative nodal price, can take any value less than or equal to $-\alpha$.
Figure~\ref{fig:ex3_2bus} illustrates the feasible set after eliminating $q_2$. The upper voltage constraint is tangential to the \emph{nose curve}~\cite{cutsem_voltage_1998} spanned by the PF equations which illustrates that LICQ does not hold.
\end{example}

\begin{example}\label{ex:pf2}
We illustrate a case where Lagrange multipliers fail to exist altogether. Consider the same 2-bus power grid as in Example~\ref{ex:pf1}. The operational constraints are given by
\begin{align*}
	 h(x) &= q_2 + \alpha \left(p_2 - p^\text{L}_2 -  v_2^{1/2}\right) = 0 \\
	 g(x)& = p_2^2 + q_2^2 - \overline{s}^2 \leq 0
\end{align*}
where $\alpha, p^\text{L}_2$ and $\overline{s}$ are parameters. The function $h$ describes a non-trivial load model with a fixed active load component $p^\text{L}_2$, an ``exponential load'' component~\cite{cutsem_voltage_1998} combined with a fixed power factor on the remaining load. The constraint $g$ describes the limit on the apparent power consumed at bus~2.

Again, we set $v_1 = 1$ and $\theta_1 = 0$. Further, we eliminate variables and constraints except $v_2$ and $\theta_2$.\footnote{Given $v_2, \theta_2$, one can always compute $p_1, p_2, q_1$, $q_2$ according to~\eqref{eq:ex3_pf}.} Both variables are unconstrained except for the operational constraints
\begin{align*}
		h(v_2, \theta_2) &= v_2^2 + v_2 \left( \alpha \sin \theta_2 - \cos \theta_2 \right) - \alpha (v_2^{1/2} + p_2^\text{L} ) = 0 \\
		g(v_2, \theta_2) &= v_2^2 \left( v_2^2 - 2 v_2 \cos \theta_2 +1 \right) - \overline{s}^2 \leq 0\,.
\end{align*}

For parameter values
$\alpha = \tfrac{2}{9}(\sqrt{3} + 3)$, $ \overline{s}^2 = 2- \sqrt{3}$, and $
		p^\text{L}_2 = -\tfrac{1}{8}(15\sqrt{3} - 23)
$ we obtain the feasible domain shown in Figure~\ref{fig:ex4_2bus}.
At the point $(v_2^\star,\theta_2^\star)=(1, \frac{\pi}{6})$ the LICQ fails to hold. In addition, the curves cross over at $(v_2^\star,\theta_2^\star)$. This implies that for an appropriate cost function $f$, no multipliers exist that satisfy~\eqref{eq:first-order} even if the cross-over point is a minimizer $(v_2^\star,\theta_2^\star)$. Namely, if $\nabla f(v_2^\star,\theta_2^\star)$ does not lie in the space spanned by $\nabla h (v_2^\star,\theta_2^\star)$ and $\nabla g(v_2^\star,\theta_2^\star)$ the KKT condition~\eqref{eq:first-order} cannot be satisfied by any $\mu$ and $\lambda$.

\begin{figure}[htb]
		\centering
		\subfloat[]{\label{fig:ex3_2bus}
				\includegraphics[width=0.47\linewidth]{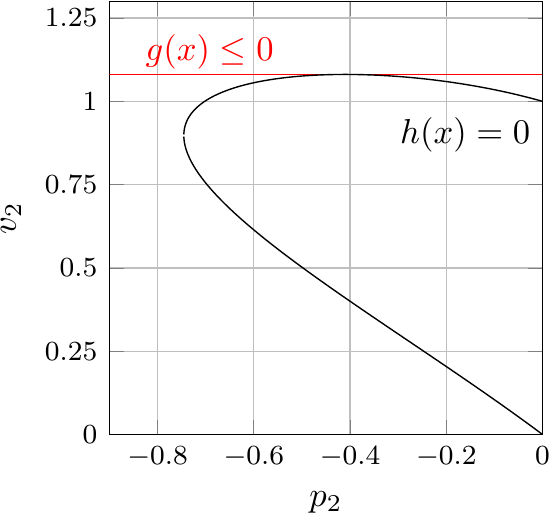}
		}
		\subfloat[]{\label{fig:ex4_2bus}
				\includegraphics[width=0.47\linewidth]{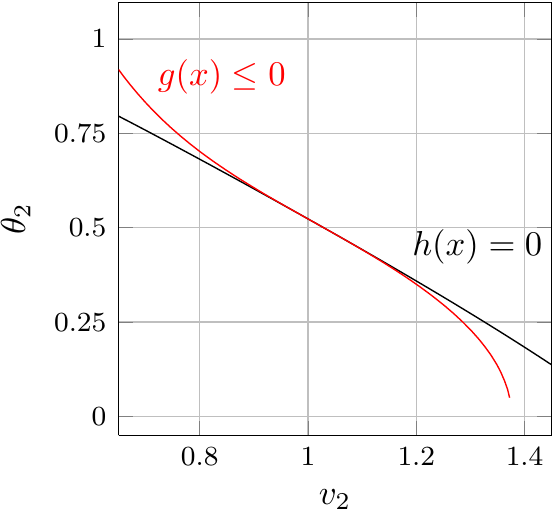}
		}
		\caption{Representation of OPF constraints that do not meet the LICQ, from Example~\ref{ex:pf1} and Example~\ref{ex:pf2}}
		\label{fig:ex12}
\end{figure}

\end{example}

\section{Perturbation Models for AC OPF}\label{sec:pert}

We use Theorem~\ref{thm:base1} to show that the feasible set of AC OPF problems satisfy the LICQ almost certainly. For this, we introduce the notion of a \emph{perturbation model} defined as a hypothetical parametrization of the AC OPF problem. In\,particular, we assume that the AC OPF problem is a generic instance of a parametrized class of problems. Consequently, if the perturbation model satisfies the requirements of Theorem~\ref{thm:base1}, namely if rank of the power flow equations as a function of the perturbation is large enough, then LICQ is almost certainly satisfied at the solutions of AC OPF\@.

In the following, we show that plausible perturbation models are available for AC OPF\@. These models are realistic in the sense that they assume variations in problem parameters that are naturally fluctuating or uncertain such as loads or grid parameters.
Consequently, the genericity of LICQ is guaranteed under conditions encountered when formulating an AC OPF problem from noisy measurements, e.g., in online applications. 

\subsection{Perturbation of fixed loads}

We claim that by perturbing real and reactive loads at every node of a $N$-bus power system, we obtain an effective perturbation model. For this, define the parameter vector
\begin{equation} \label{eq:pert1}
		\xi^\text{L} = \begin{bmatrix} p^\text{L} \\ q^\text{L} \end{bmatrix} \in \Xi^\text{L} \subseteq \mathbb{R}^{2N}
\end{equation}
where $\Xi^\text{L}$ is an open, nonempty set of load configurations. By defining $\widetilde{h}(x, \xi^\text{L}) = F(x)$ we turn the AC OPF problem~\eqref{eq:opf} into a parametrized optimization problem of the form~\eqref{eq:opt_param} with $n= 4N$, $\widetilde{I}= 2N$, $\widetilde{J}=0$ and $I,J$ given by the number of operational equality and inequality constraints.

\begin{prop}\label{prop:main1}
Consider a parametrized AC OPF problem of the form $\mathsf{P}(\xi^\text{L})$ with $\xi^\text{L}$ defined as in~\eqref{eq:pert1} and $\widetilde{h}(x, \xi^\text{L}) = 0$ denoting the parametrized AC power flow equations. Further, assume that the operational constraints $g$ and $h$ are of class $C^r$ with $r > 2N$ for all $x \in \mathcal{X}$ and that Assumption~\ref{ass1} holds.
Then, for almost every load profile $\xi^\text{L} \in \Xi^\text{L}$ the LICQ holds for every $x$ that is feasible for $\mathsf{P}(\xi^\text{L})$.
\end{prop}

For practical purposes Proposition~\ref{prop:main1} implies that an AC OPF problem for which fixed loads at every bus are randomly sampled (from a continuous distribution) satisfy the LICQ everywhere with probability one.
In Examples~\ref{ex:pf1} and~\ref{ex:pf2} this may be seen as a deformation of the curves $h(x)=0$ in Figure~\ref{fig:ex12} which eliminates the tangency between the constraints.

\begin{proof}
We apply Theorem~\ref{thm:base1} directly. The requirement on the differentiability only applies to $g$ and $h$ since $\widetilde{h}$ is smooth. The fact that $n= 4N$ and $\widetilde{I}= 2N$ establishes the requirement that $g$ and $h$ have to be $r$-times continuously differentiable for all $x \in \mathcal{X}$. Finally,  $\rank \frac{d\widetilde{h}}{d\xi}(x, \xi^\text{L}) = 2N$ since $\frac{dh}{d\xi}(x, \xi^\text{L})$ is the $2N$-dimensional identity map.
\end{proof}

\subsection{Perturbation of shunt admittances}

Next we show that the perturbation of shunt admittances at every bus also guarantees the genericity of LICQ\@. For every node $k$ in the power network define the lumped shunt admittance $\widetilde{y}^\text{sh}_k = \widetilde{g}^\text{sh}_k + j \widetilde{b}^\text{sh}_k := \widetilde{y}^\text{sh}_k + \frac{1}{2} \sum_{i \in N(k)}{y^{\text{sh}}_{ki}} \in \mathbb{C}$ and define the associated parameter vector
\begin{equation} \label{eq:pert2}
	 \xi^{\text{sh}} = \begin{bmatrix} \widetilde{g}^\text{sh} \\ \widetilde{b}^\text{sh} \end{bmatrix} \in \Xi^\text{sh} \subseteq \mathbb{R}^{2N}
\end{equation}
If $\xi^{\text{sh}}$ varies over an open, nonempty set $\Xi^\text{sh}$, the AC OPF problem~\eqref{eq:opf} is again a parametrized optimization problem of the the form~\eqref{eq:opt_param} with $n= 4N$, $\widetilde{I}= 2N$, $\widetilde{J}=0$ as before.

\begin{prop}\label{prop:main2}
Consider a parametrized AC OPF problem of the form~\eqref{eq:opt_param} with $\xi^\text{sh}$ defined as in~\eqref{eq:pert2} and $\widetilde{h}(x, \xi^\text{sh}) = 0$ denoting the parametrized AC power flow equations. Further, assume that the operational constraints $g$ and $h$ are of class $C^r$ with $r > 2N$ for all $x \in \mathcal{X}$ and that Assumption~\ref{ass1} holds.
Furthermore, assume that $g$ and $h$ are such that for every $x \in \mathcal{X}$ we have $v_k > 0$ for all $k=1, \ldots, n$.
Then, for almost every value of $\xi \in \Xi$ the LICQ holds for every $x$ that is feasible for~\eqref{eq:opt_param}.
\end{prop}

\begin{proof}
As in the proof of Proposition~\ref{prop:main1}, we apply Theorem~\ref{thm:base1}. The crucial difference is that the derivative of $\widetilde{h}$ with respect to $\xi^\text{sh}$ is given by
\begin{equation*}
		\frac{d\widetilde{h}}{d\xi}(x, \xi^\text{sh}) = \begin{bmatrix} -\diag(v^2) & 0 \\ 0 & -\diag(v^2) \end{bmatrix}
\end{equation*}
where $v^2$ denotes the componentwise square of $v$. Therefore $\frac{d\widetilde{h}}{d\xi}(x, \xi^\text{sh})$ depends on the fact that none of the components of $v$ is zero in which case its rank is $2N$ as required.
\end{proof}

In comparison to Proposition~\ref{prop:main1}, the additional assumption that all voltage magnitudes have to be nonzero is easily assured if $g$ incorporates lower voltage bounds.
Note that Propositions~\ref{prop:main1} and~\ref{prop:main2} can be combined by requiring that at every node either a fixed load or a shunt is present.

\subsection{Perturbation of line parameters}

Finally, we ask whether in the absence of shunt admittances the perturbation of line parameters itself is enough to guarantee the genericity of LICQ\@. This is not the case. Namely, the following example reveals a \emph{structural} degeneracy of the PF equations without shunt admittances which cannot be mitigated by perturbation of line parameters only.

\begin{example}\label{ex:degen}
		Consider a network without shunt admittances, i.e., $y_{kl}^\text{sh} = y_k^\text{sh} = 0$ for all $k,l \in \mathbf{N}$. Then, consider the no-load configuration $p^\text{L} = q^\text{L} = 0$ of the power system with flat voltage profile $v e^{j\theta} = \mathds{1}$ (the vector of all ones). It is well known that $\mathds{1}$ is an eigenvector of $Y$. Therefore, perturbing the line admittances $y_{kl}$ does not change the operating point and the requirement for Theorem~\ref{thm:base1} is not satisfied.
\end{example}

\section{Conclusion}

We have discussed and illustrated that for AC OPF problems the failure to satisfy CQs, can lead to ill-defined Lagrange multipliers. However,
we have shown that under the weak technical (and in practice often valid) assumption of a perturbation model, the LICQ holds generically at all solutions giving rise to a unique set of Lagrange multipliers. Thereby, we effectively eliminate the need for the boilerplate assumption that CQs always hold for AC OPF\@. This guarantee is in full effect when problem parameters such as loads are randomly sampled, e.g., for online applications that derive problem parameters from noisy measurements.

\appendix

We provide a condensed proof of Theorem~\ref{thm:base1} based on Thom's Transversality Theorem.
The additional terms and definitions in this section are independent of the main body of the paper and follow~\cite{hirsch_differential_1976}.

\begin{definition}\label{def:transv}
		Let $\calF', \calN$ be manifolds and $\calM \subset \calN$ be a submanifold of $\calN$. A map $\Phi: \calF' \rightarrow \calN$ is \emph{transverse to $\calM$} if $T_{\Phi(x)} \calN = T_{\Phi(x)} \calM + D\Phi(x)(T_x \calF')$ for all $x$ such that $\Phi(x) \in \calM$ where $D\Phi$ denotes the differential of $\Phi$.
\end{definition}

In other words, the tangent space of $\calN$ is spanned by the tangent space of $\calM$ and the image of the tangent space of $\calF'$. The following classical result can be found in~\cite{hirsch_differential_1976}.

\begin{thm}[Thom's Parametric Transversality Theorem]\label{thm:thom}
		Let $\calF, \calN$ be $C^r$ manifolds and let $\calM \subset \calN$ be a $C^r$ submanifold of $\calN$. Further, let $\Phi: \calF \times \Xi \rightarrow \calN$ be of class $C^r$ where $\Xi \subset \mathbb{R}^k$ is an open set. Assume that $r > \max \{ 0, \dim \calF + \dim \calM - \dim \calN \}$ and that $\Phi(\cdot, \cdot)$ is transverse to $\calM$ (as a function of both arguments).
		Then, $\Phi(\cdot, \xi): \calF \rightarrow \calN$ is transverse to $\calM$ (as a function of the first argument) for almost all $\xi \in \Xi$.
\end{thm}

Informally, Theorem~\ref{thm:thom} states that the manifold $\calM$ and the image of $\Phi(\cdot, \xi)$ are almost never tangent, i.e., they are transverse for almost all $\xi \in \Xi$. Theorem~\ref{thm:base1} is a generalization in so far as the feasible set of an optimization problem is delimited by multiple manifolds (i.e., \emph{faces}) that must intersect transversely in order to satisfy LICQ\@.

To show Theorem~\ref{thm:base1}, we first show that the set $\mathcal{X}$ described by the fixed constraints can indeed be partitioned into a finite number of \emph{faces} defined by the active inequality constraints. We can then apply Theorem~\ref{thm:thom} to each face individually.

\begin{lemma}\label{lem:finite}
If~\eqref{eq:opt_param} satisfies Assumption~\ref{ass1} and $g,h$ are of class $C^r$ on $\mathcal{X}$, then $\mathcal{X}$ can be partitioned into finitely many faces \begin{align*}
		\overline{\mathcal{X}}_{\mathbf{J}'} := \left\lbrace x \in \mathcal{X} \, | \, h(x) = 0, \, g_{\mathbf{J}'}(x) = 0 ,\, g_{\mathbf{J} \setminus \mathbf{J}'}(x) < 0\right\rbrace
\end{align*}
where $\mathbf{J}' \subset \mathbf{J} := \{1, \ldots, J\}$. Further, for every $\mathbf{J}'$, the set $\overline{\mathcal{X}}_{\mathbf{J}'}$ is a $C^r$ submanifold of $\mathbb{R}^n$ of dimension $n - I - J'$ where $J' := | \mathbf{J}' |$.
\end{lemma}

\begin{proof} Consider the open set (and hence open submanifold) given by $\mathcal{H} := \{ x \, | \, g_{\mathbf{J} \setminus \mathbf{J}'}(x) < 0 \}$.
 By Assumption~\ref{ass1} the map $x \mapsto (h(x),\, g_{\mathbf{J}'}(x))$ defined from
$\mathcal{H}$ to $\mathbb{R}^I \times \mathbb{R}^{J'}$
 evaluates to zero and has rank $I + J'$ for all $x \in \overline{\mathcal{X}}_{\mathbf{J}'}$. It follows from~\cite[Thm 3.2]{hirsch_differential_1976} that $\mathcal{X}_{\mathbf{J}'}$ is a $C^r$-submanifold of $\mathcal{H}$ (and by extension of $\mathbb{R}^n$) of dimension $n - (I + J')$. Since $J$ is finite, there are only finitely many faces.
\end{proof}

\begin{lemma}\label{lem:face}
Let~\eqref{eq:opt_param} satisfy Assumption~\ref{ass1} and $g,h$ be of class $C^r$ on $\mathcal{X}$ where $r > \max \{ 0, n - I - \widetilde{I}\}$. Consider $\mathbf{J}' \subset \{1, \ldots, J\}$, the associated face $\overline{\mathcal{X}}_{\mathbf{J}'}$ and the map
\begin{align*}
		\Phi: \overline{\mathcal{X}}_{\mathbf{J}'} \times \Xi & \rightarrow \mathbb{R}^{\widetilde{I}} \times \mathbb{R}^{\widetilde{J}} \\
		(x,\xi) & \mapsto (\widetilde{h}(x, \xi), \widetilde{g}(x, \xi)) \,.
\end{align*}
If $\Phi(x, \cdot)$ has rank $\widetilde{I}+\widetilde{J}$ for all $x \in \mathcal{X}$ and all $\xi \in \Xi$, then $h, g_{\mathbf{J}'}, \widetilde{h}$ and $\widetilde{g}$ satisfy the LICQ for almost all $\xi \in \Xi$ and all $x \in \overline{\mathcal{X}}_{\mathbf{J}'}$.
\end{lemma}

\begin{proof}
Consider a subset $\widetilde{\mathbf{J}}' \subset \{1, \ldots, \widetilde{J}\}$ and let $\widetilde{J}' := | \widetilde{\mathbf{J}}' |$. We define the submanifold of dimension $\widetilde{J} - \widetilde{J}'$ given as
\begin{align*}
		\calM_{\widetilde{\mathbf{J}}'} = \left\lbrace (0, y) \in \mathbb{R}^{\widetilde{I}} \times \mathbb{R}^{\widetilde{J}} \, \middle | \, \forall j \in \widetilde{\mathbf{J}}': \, y_j = 0 \right\rbrace \,.
\end{align*}
Since $\Phi(x, \cdot)$ has rank $\widetilde{I}+\widetilde{J}$ it follows that $\Phi(\cdot, \cdot)$ has rank $\widetilde{I}+\widetilde{J}$ for all $x$ and $\xi$. Then, $\Phi(\cdot, \cdot)$ is transverse to $\calM_{\widetilde{\mathbf{J}}'}$ according to Definition~\ref{def:transv} with $\mathcal{N} = \mathbb{R}^{\widetilde{I}} \times \mathbb{R}^{\widetilde{J}}$, $\mathcal{M} = \calM_{\widetilde{\mathbf{J}}'}$ and $\mathcal{F}' = \overline{\mathcal{X}}_{\mathbf{J}'} \times \Xi$.
This is independent of the choice of $\widetilde{\mathbf{J}}'$.

Next we apply Theorem~\ref{thm:thom} with $\mathcal{F} = \overline{\mathcal{X}}_{\mathbf{J}'}$ and $\mathcal{N}, \mathcal{M}$, $\Phi$ as above. The degree of differentiability $r$ needs to satisfy
\begin{align*}
		r & > \max \{0, \dim \mathcal{X}_{J'} + \dim \calM_\mathcal{J'} - \dim \mathbb{R}^{\widetilde{I}} \times \mathbb{R}^{\widetilde{J}} \} \\
		& =\max \{0, (n- I - J') + (\widetilde{J} - \widetilde{J}') - (\widetilde{I} +\widetilde{J}) \}
		\end{align*}
Clearly, $r > \max \{0, n- I - \widetilde{I}\}$ is sufficient. Thus, Theorem~\ref{thm:thom} states that the map
\begin{align*} \Phi(\cdot, \xi): \overline{\mathcal{X}}_{\mathbf{J}'} & \rightarrow \mathbb{R}^{\widetilde{I}} \times \mathbb{R}^{\widetilde{J}}\\
				 x &\mapsto (\widetilde{h}(x, \xi), \, \widetilde{g}(x, \xi))
\end{align*}
is transverse to $\calM_{\widetilde{\mathbf{J}}'}$ for almost all $\xi$. Finally, we need to show that $\Phi(\cdot, \xi)$ being transverse to $\calM_{\widetilde{\mathbf{J}}'}$ for a given $\xi$ implies that
\begin{align*}
		x \mapsto \left( h(x), \, g_{\mathbf{J}'}(x),\, \widetilde{h}(x, \xi),\, \widetilde{g}_{\widetilde{\mathbf{J}}'}(x, \xi) \right)
\end{align*}
has full rank for all $x \in \overline{\mathcal{X}}_{\mathbf{J}'}$.

Since $\mathbb{R}^{\widetilde{I}} \times \mathbb{R}^{\widetilde{J}}$ and $\calM_{\widetilde{\mathbf{J}}'}$ are isomorphic to their tangent space, the definition of transversality applied to $\Phi(\cdot, \xi)$ yields
\begin{align*}
		\mathbb{R}^{\widetilde{I}} \times \mathbb{R}^{\widetilde{J}} = \calM_{\widetilde{\mathbf{J}}'} +
		\begin{bmatrix}
			 \nabla \widetilde{h} (x) \\
			 \nabla \widetilde{g} (x)
		\end{bmatrix}
		\ker \begin{bmatrix}
			 \nabla h (x) \\
			 \nabla {g}_{\mathbf{J}' (x)}
		\end{bmatrix}
\end{align*}
for all $x \in \overline{\mathcal{X}}_{\mathbf{J}'}$ which is equivalent to
\begin{equation} \label{eq:transv2}
		\mathbb{R}^{\widetilde{I}} \times \mathbb{R}^{\widetilde{J}'} =
		\begin{bmatrix}
			 \nabla_x \widetilde{h} (x, \xi) \\
			 \nabla_x \widetilde{g}_{\widetilde{\mathbf{J}}'} (x, \xi)
		\end{bmatrix}
		\ker \begin{bmatrix}
			 \nabla h (x) \\
			 \nabla {g}_{\mathbf{J}'}(x)
		\end{bmatrix}
\end{equation}
by eliminating the subspace spanned by $\calM_{\widetilde{\mathbf{J}}'}$. It follows that
\begin{align*}
		\ker \begin{bmatrix}
			 \nabla_x \widetilde{h} (x, \xi) \\
			 \nabla_x \widetilde{g}_{\widetilde{\mathbf{J}}'} (x, \xi)
		\end{bmatrix} ^\perp \subset
		\ker \begin{bmatrix}
			 \nabla h (x) \\
			 \nabla {g}_{{\mathbf{J}}'} (x) \\
		\end{bmatrix}
\end{align*}
and consequently,
using the fundamental theorem of linear algebra, we get
\begin{equation*}
		\range \begin{bmatrix}
			 \nabla_x \widetilde{h} (x, \xi) \\
			 \nabla_x \widetilde{g}_{\widetilde{\mathbf{J}}'} (x, \xi)
		\end{bmatrix}^T \perp
		\range \begin{bmatrix}
			 \nabla h (x) \\
			 \nabla {g}_{{\mathbf{J}}'} (x) \\
		\end{bmatrix}^T \,.
\end{equation*}

Together with the fact that $[
			 \nabla h^T (x) \ 
			 \nabla {g}^T_{{\mathbf{J}}'} (x)
		]$ has rank $I + J'$ (by Assumption~\ref{ass1}) and that $[
			 \nabla_x \widetilde{h}^T (x, \xi) \ 
			 \nabla_x \widetilde{g}^T_{\widetilde{\mathbf{J}}'} (x, \xi)
		]$ has rank $\widetilde{I} + \widetilde{J}'$ (by~\eqref{eq:transv2}) it follows that $\nabla g_{\mathbf{J}'} (x), \nabla h (x), \nabla_x \widetilde{g}_{\widetilde{\mathbf{J}}'} (x, \xi)$ and $\nabla_x \widetilde{h} (x, \xi)$ are linearly independent for all $x \in \overline{\mathcal{X}}_{\mathbf{J}'}$ and all $\xi \in \Xi \setminus \Xi_{\widetilde{\mathbf{J}}'}$ where $\Xi_{\widetilde{\mathbf{J}}'}$ has zero measure.
Finally, since the finite union over all $\widetilde{\mathbf{J}}'$ of $\Xi_{\widetilde{\mathbf{J}}'}$ is a zero measure set, we conclude that the LICQ hold generically for a given face $\mathcal{X}_{\mathbf{J}}'$.
\end{proof}

By combining Lemmas~\ref{lem:finite} and~\ref{lem:face} we finally prove Theorem~\ref{thm:base1}, as the finite union over all $\mathbf{J}' \subset \mathbf{J}$ of the zero-measure subsets of $\Xi$ on which LICQ fails is again a zero-measure set. 

\nocite{*}
\bibliographystyle{IEEEtran}
\bibliography{bibliography_static}

\end{document}